\documentclass[12pt]{article}
\usepackage{amsmath, amsthm, amssymb,enumerate,verbatim,authblk, fullpage}
\usepackage{tikz}

\usepackage{hyperref}

\theoremstyle{plain}
\newtheorem{thm}{Theorem}
\newtheorem{lem}[thm]{Lemma}
\newtheorem{prop}[thm]{Proposition}
\newtheorem{conj}[thm]{Conjecture}
\newtheorem{claim}[thm]{Claim}

\newtheorem{prob}[thm]{Problem}

\theoremstyle{definition}

\newtheorem{defn}[thm]{Definition}
\newtheorem{cons}[thm]{Construction}
\newtheorem{ex}[thm]{Example}

\newtheorem{rem}[thm]{Remark}

\newtheoremstyle{case}{}{}{\normalfont}{}{\itshape}{\normalfont:}{ }{}

\theoremstyle{case}
\newtheorem{case}{Case}

\title{On Saturated $k$-Sperner Systems}
\author{Natasha Morrison}
\author{Jonathan A. Noel}
\author{Alex Scott}

\affil{\normalsize{Mathematical Institute, University of Oxford, Woodstock Road, Oxford OX2 6GG, UK.}}
\affil{\texttt{\{morrison,noel,scott\}@maths.ox.ac.uk}}

\newcommand{\sml}[1]{\mathcal{#1}^{\text{{\rm small}}}}
\newcommand{\lrg}[1]{\mathcal{#1}^{\text{{\rm large}}}}

\newcommand{\sat}{\operatorname{sat}}
\newcommand{\wsat}{\operatorname{osat}}

\begin{document}

\maketitle

\begin{abstract}
Given a set $X$, a collection $\mathcal{F}\subseteq\mathcal{P}(X)$ is said to be \emph{$k$-Sperner} if it does not contain a chain of length $k+1$ under set inclusion and it is \emph{saturated} if it is maximal with respect to this property. Gerbner et al.~\cite{gerb} conjectured that, if $|X|$ is sufficiently large with respect to $k$, then the minimum size of a saturated $k$-Sperner system $\mathcal{F}\subseteq\mathcal{P}(X)$ is $2^{k-1}$. We disprove this conjecture by showing that there exists $\varepsilon>0$ such that for every $k$ and $|X| \geq n_0(k)$ there exists a saturated $k$-Sperner system $\mathcal{F}\subseteq\mathcal{P}(X)$ with cardinality at most $2^{(1-\varepsilon)k}$. 

A collection $\mathcal{F}\subseteq \mathcal{P}(X)$ is said to be an \emph{oversaturated $k$-Sperner system} if, for every $S\in\mathcal{P}(X)\setminus\mathcal{F}$, $\mathcal{F}\cup\{S\}$ contains more chains of length $k+1$ than $\mathcal{F}$. Gerbner et al.~\cite{gerb} proved that, if $|X|\geq k$, then the smallest such collection contains between $2^{k/2-1}$ and $O\left(\frac{\log{k}}{k}2^k\right)$ elements. We show that if $|X|\geq k^2+k$, then the lower bound is best possible, up to a polynomial factor. 

  % keywords are optional
%  \bigskip\noindent \textbf{Keywords:} minimum saturation; set systems; antichains
\end{abstract}

\section{Introduction}

Given a set $X$, a collection $\mathcal{F}\subseteq \mathcal{P}(X)$ is a \emph{Sperner system} or an \emph{antichain} if there do not exist $A,B\in\mathcal{F}$ such that $A\subsetneq B$. More generally, a collection $\mathcal{F}\subseteq \mathcal{P}(X)$ is a \emph{k-Sperner system} if there does not exist a subcollection $\{A_1,\dots,A_{k+1}\}\subseteq\mathcal{F}$ such that $A_1\subsetneq \dots\subsetneq A_{k+1}$. Such a subcollection $\{A_1,\dots,A_{k+1}\}$ is called a \emph{$(k+1)$-chain}. We say that a $k$-Sperner system is \emph{saturated} if, for every $S\in\mathcal{P}(X)\setminus \mathcal{F}$, we have that $\mathcal{F}\cup\{S\}$ contains a $(k+1)$-chain. A collection $\mathcal{F}\subseteq \mathcal{P}(X)$ is an \emph{oversaturated $k$-Sperner system}\footnote{In~\cite{gerb}, this is called a \emph{weakly saturated $k$-Sperner system}. Since there is another notion of weak saturation in the literature (see, for instance, 
Bollob\'{a}s~\cite{weakly}), we have chosen to use a different term to avoid possible 
confusion.} if, for every $S\in\mathcal{P}(X)\setminus\mathcal{F}$, we have that the number of $(k+1)$-chains in $\mathcal{F}\cup\{S\}$ is greater than the number of $(k+1)$-chains in $\mathcal{F}$. Thus, $\mathcal{F}\subseteq\mathcal{P}(X)$ is a saturated $k$-Sperner system if and only if it is an oversaturated $k$-Sperner system that does not contain a $(k+1)$-chain.

For a set $X$ of cardinality $n$, the problem of determining the maximum size of a saturated $k$-Sperner system in $\mathcal{P}(X)$ is well understood. In the case $k=1$, Sperner's Theorem~\cite{Sperner} (see also~\cite{WhiteBook}), says that every antichain in $\mathcal{P}(X)$ contains at most $\binom{n}{\left\lfloor n/2\right\rfloor}$ elements, and this bound is attained by the collection consisting of all subsets of $X$ with cardinality $\left\lfloor n/2\right\rfloor$. Erd\H{o}s~\cite{LittlewoodOfford} generalised Sperner's Theorem by proving that the largest size of a $k$-Sperner system in $\mathcal{P}(X)$ is the sum of the $k$ largest binomial coefficients $\binom{n}{i}$. In this paper, we are interested in determining the minimum size of a saturated $k$-Sperner system or an oversaturated $k$-Sperner system in $\mathcal{P}(X)$. These problems were first studied by Gerbner, Keszegh, Lemons, Palmer, P{\'a}lv{\"o}lgyi and Patk{\'o}s~\cite{gerb}.

Given integers $n$ and $k$, let $\sat(n,k)$ denote the minimum size of a saturated $k$-Sperner system in $\mathcal{P}(X)$ where $|X|=n$. It was shown in~\cite{gerb} that $\sat(n,k)=\sat(m,k)$ if $n$ and $m$ are sufficiently large with respect to $k$. We can therefore define
\[\sat(k):=\lim_{n\to\infty}\sat(n,k).\]
We are motivated by the following conjecture of~\cite{gerb}.

\begin{conj}[Gerbner et al.~\cite{gerb}]
\label{mainconj}
For all $k$, $\sat(k)=2^{k-1}$. 
\end{conj}

Gerbner et al.~\cite{gerb} observed that their conjecture is true for $k=1,2,3$. They also proved that $2^{k/2-1}\leq \sat(k)\leq 2^{k-1}$ for all $k$, where the upper bound is implied by the following construction.

\begin{cons}[Gerbner et al.~\cite{gerb}]
\label{powerset}
Let $Y$ be a set such that $|Y| = k-2$ and let $H$ be a non-empty set disjoint from $Y$. Let $X = Y \cup H$ and define 
\[\mathcal{G}:=\mathcal{P}(Y)\cup\{S\cup H: S\in\mathcal{P}(Y)\}.\]
It is easily verified that $\mathcal{G} \subseteq \mathcal{P}(X)$ is a saturated $k$-Sperner system of cardinality $2^{k-1}$.
\end{cons}

In this paper, we disprove Conjecture~\ref{mainconj} by establishing the following:

\begin{thm}
\label{epsilon}
There exists $\varepsilon>0$ such that, for all $k$, $\sat(k) \le 2^{(1-\varepsilon)k}$.
\end{thm}

We remark that the value of $\varepsilon$ that can be deduced from our proof is approximately $\left(1 - \frac{\log_{2}\left(15\right)}{4}\right)\approx 0.023277$. The proof of Theorem~\ref{epsilon} comes in two parts. First, we give an infinite family of saturated $6$-Sperner systems of cardinality $30$ which shows that $\sat(6) \leq 30< 2^5$. We then provide a method which, under certain conditions, allows us to combine a saturated $k_1$-Sperner system of small order and a saturated $k_2$-Sperner system of small order to obtain a saturated $(k_1+k_2-2)$-Sperner system of small order. By repeatedly applying this method, we are able to prove Theorem~\ref{epsilon} for general $k$. As it turns out, our method yields the bound $\sat(k) < 2^{k-1}$ for every $k\geq6$. For completeness, we will prove that $\sat(k)=2^{k-1}$ for $k\leq 5$, and so $k=6$ is the first value of $k$ for which Conjecture~\ref{mainconj} is false.

Similar techniques show that $\sat(k)$ satisfies a submultiplicativity condition, which leads to the following result.

\begin{thm}\label{submult}
For $\varepsilon$ as in Theorem~\ref{epsilon}, there exists $c\in[1/2,1-\varepsilon]$ such that $\sat(k) = 2^{(1+o(1))ck}$.
\end{thm}

Naturally, we wonder about the correct value of $c$ in Theorem~\ref{submult}.

\begin{prob}
\label{problem}
Determine the constant $c$ for which $\sat(k) = 2^{(1+o(1))ck}$.
\end{prob}

We are also interested in oversaturated $k$-Sperner systems. Given integers $n$ and $k$, let $\wsat(n,k)$ denote the minimum size of an oversaturated $k$-Sperner system in $\mathcal{P}(X)$ where $|X|=n$. As we will prove in Lemma~\ref{satk}, $\wsat(n,k)=\wsat(m,k)$ provided that $n$ and $m$ are sufficiently large with respect to $k$. Similarly to $\sat(k)$, we define $\wsat(k):=\lim_{n\to \infty}\wsat(n,k)$. Gerbner et al.~\cite{gerb} proved that if $|X|\geq k$, then an oversaturated $k$-Sperner system in $\mathcal{P}(X)$ of minimum size has between $2^{k/2 - 1}$ and $O\left(\frac{\log(k)}{k}2^k\right)$ elements. Together with Lemma~\ref{satk}, this implies
\[2^{k/2-1}\leq \wsat(k)\leq O\left(\frac{\log(k)}{k}2^k\right).\]
We show that the lower bound gives the correct asymptotic behaviour, up to a polynomial factor. 

\begin{thm}\label{wsatb}
For every integer $k$ and set $X$ with $|X|\geq k^2+k$ there exists an oversaturated $k$-Sperner system $\mathcal{F}\subseteq \mathcal{P}(X)$ such that $|\mathcal{F}| = O\left(k^5 2^{k/2}\right)$. In particular, 
\[\wsat(k) = 2^{(1/2 + o(1))k}.\]
\end{thm}

In Section~\ref{prelims}, we prove some preliminary results which will be used throughout the paper. In particular, we provide conditions under which a saturated $k$-Sperner system can be decomposed into or constructed from a sequence of $k$ disjoint saturated antichains. In Section~\ref{comb} we show that certain types of saturated $k_1$-Sperner and $k_2$-Sperner systems can be combined to produce a saturated $(k_1+k_2-2)$-Sperner system, and use this to prove Theorems~\ref{epsilon} and~\ref{submult}. Finally, in Section~\ref{wsat}, we give a probabilistic construction of oversaturated $k$-Sperner systems of small cardinality, thereby proving Theorem~\ref{wsatb}.

Minimum saturation has been studied extensively in the context of graphs~\cite{completepartite,BollobasEHM,C5,Kst,hypercube,KaszTuz,C4,Wessel,Wessel2} and  hypergraphs~\cite{runif,OlegHyper2,rstars,OlegHyper}. Such problems are typically of the following form: for a fixed (hyper)graph $H$, determine the minimum size of a (hyper)graph $G$ on $n$ vertices which does not contain a copy of $H$ and for which adding any edge $e \notin G$, yields a (hyper)graph which contains a copy of $H$. This line of research was first initiated by Zykov~\cite{Zykov} and Erd\H{o}s, Hajnal and Moon~\cite{EHM}. For more background on minimum saturation problems for graphs, we refer the reader to the survey of Faudree, Faudree and Schmitt~\cite{satsurvey}.

\section{Preliminaries}
\label{prelims}

Given a collection $\mathcal{F}\subseteq \mathcal{P}(X)$, we say that a set $A\subseteq X$ is an \emph{atom} for $\mathcal{F}$ if $A$ is maximal with respect to the property that
\begin{equation}\label{hom} \text{for every set $S\in\mathcal{F}$, $S\cap A\in\{\emptyset,A\}$.}\end{equation}
We say that an atom $A$ with $|A|\geq2$ is \emph{homogeneous} for $\mathcal{F}$. Gerbner et al.~\cite{gerb} proved that if $n,m$ are sufficiently large with respect to $k$, then $\sat(n,k)=\sat(m,k)$. Using a similar approach, we extend this result to $\wsat(n,k)$. 

\begin{lem}
\label{satk}
Fix $k$. If $n,m> 2^{2^{k-1}}$, then $\sat(n,k)=\sat(m,k)$ and $\wsat(n,k)=\wsat(m,k)$.
\end{lem}

\begin{proof}
Fix $n>2^{2^{k-1}}$ and let $X$ be a set of cardinality $n$. Suppose that $\mathcal{F}\subseteq \mathcal{P}(X)$ is an oversaturated $k$-Sperner system of cardinality at most $2^{k-1}$. We know that such a family exists by Construction \ref{powerset}. We will show that, for sets $X_1$ and $X_2$ such that $|X_1|=n-1$ and $|X_2|=n+1$, there exists $\mathcal{F}_1\subseteq \mathcal{P}(X_1)$ and $\mathcal{F}_2\subseteq \mathcal{P}(X_2)$ such that
\begin{enumerate}[(a)]
\item\label{samecard} $|\mathcal{F}_1|=|\mathcal{F}_2|=|\mathcal{F}|$,
\item \label{samechains} $\mathcal{F}_1$ and $\mathcal{F}_2$ have the same number of $(k+1)$-chains as $\mathcal{F}$,
\item\label{stillwsat} $\mathcal{F}_1$ and $\mathcal{F}_2$ are oversaturated $k$-Sperner systems.
\end{enumerate}
We observe that this is enough to prove the lemma. Indeed, by taking $\mathcal{F}$ to be a saturated $k$-Sperner system or an oversaturated $k$-Sperner system in $\mathcal{P}(X)$ of minimum order, we will have that 
\[\max\{\sat(n-1,k), \sat(n+1,k)\} \leq \sat(n,k)\text{ and}\]
\[\max\{\wsat(n-1,k), \wsat(n+1,k)\} \leq \wsat(n,k).\]
Since $n$ was an arbitrary integer greater than $2^{2^{k-1}}$, the result will follow by induction. 

We prove the following claim.

\begin{claim}
\label{bighom}
Given a set $X$ and a collection $\mathcal{F}\subseteq\mathcal{P}(X)$, if $|X|>2^{|\mathcal{F}|}$, then there is a homogeneous set for $\mathcal{F}$.
\end{claim}

\begin{proof}
We observe that every atom $A$ for $\mathcal{F}$ corresponds to a subcollection $\mathcal{F}_A:=\{S\in \mathcal{F}: A\subseteq S\}$ of $\mathcal{F}$ such that $\mathcal{F}_A\neq \mathcal{F}_{A'}$ whenever $A\neq A'$. This implies that the number of atoms for $\mathcal{F}$ is at most $2^{|\mathcal{F}|}$. Therefore, since $|X|> 2^{|\mathcal{F}|}$, there must be a homogeneous set $H$ for $\mathcal{F}$. 
\end{proof}

By Claim~\ref{bighom} and the fact that $|X|>2^{2^{k-1}}\geq2^{|\mathcal{F}|}$, there exists a homogeneous set $H$ for $\mathcal{F}$. Let $x_1\in H$ and $x_2\notin X$ and define $X_1 := X\setminus\{x_1\}$ and $X_2 := X\cup\{x_2\}$. Let 
\[\mathcal{F}_1 := \{S\in \mathcal{F}: S\cap H = \emptyset\} \cup \{S\setminus\{x_1\}: S\in \mathcal{F}_H\},\text{ and}\]
\[\mathcal{F}_2 := \{S\in \mathcal{F}: S\cap H = \emptyset\} \cup \{S\cup\{x_2\}: S\in \mathcal{F}_H\}.\]
Since $H$ is homogeneous for $\mathcal{F}$, there does not exist a pair of sets in $\mathcal{F}$ which differ only on $x_1$. Thus, for $i\in\{1,2\}$ there is a natural bijection from $\mathcal{F}_i$ to $\mathcal{F}$ which preserves set inclusion. Hence, (\ref{samecard}) and (\ref{samechains}) hold. Now, let $i\in\{1,2\}$ and $T_i\in\mathcal{P}(X_i)\setminus\mathcal{F}_i$ and define
\[T:=\left(T_i\setminus (H\cup\{x_2\})\right)\cup\{x_1\}.\]
Then $T\in\mathcal{P}(X)\setminus\mathcal{F}$ since $H$ is a non-singleton atom and $T\cap H=\{x_1\}$, and so there exists $A_1,\dots,A_k\in\mathcal{F}$ and $t\in\{0,\dots,k\}$ such that
\[A_1\subsetneq\dots\subsetneq A_t\subsetneq T\subsetneq A_{t+1}\subsetneq \dots\subsetneq A_k.\]
Since $T\cap H\neq H$, we must have $A_j\cap H=\emptyset$ for $j\leq t$ and so $A_1,\dots,A_t\in\mathcal{F}_i$ and $A_1\subsetneq\dots\subsetneq A_t\subsetneq T_i$. Also, since $T\cap H\neq \emptyset$, we have $A_j\cap H=H$ for $j\geq t+1$. Setting $A_j':=(A_j\cup\{x_2\})\cap X_i$, we see that $A_j'\in\mathcal{F}_i$ for $j\geq t+1$ and that $T_i\subsetneq A_{t+1}'\subsetneq \dots\subsetneq A_k'$. Thus, (\ref{stillwsat}) holds. 
\end{proof}

The rest of the results of this section are concerned with the structure of saturated $k$-Sperner systems. The next lemma, which is proved in~\cite{gerb}, implies that for any saturated $k$-Sperner system there can be at most one homogeneous set. We include a proof for completeness.

\begin{lem}[Gerbner et al.~\cite{gerb}]
\label{block}
If $\mathcal{F}\subseteq \mathcal{P}(X)$ is a saturated $k$-Sperner system and $H_1$ and $H_2$ are homogeneous for $\mathcal{F}$, then $H_1=H_2$.
\end{lem}

\begin{proof}
Suppose to the contrary that $H_1$ and $H_2$ are homogeneous for $\mathcal{F}$ and that $H_1\neq H_2$. Then, since each of $H_1$ and $H_2$ are maximal with respect to (\ref{hom}), we have that $H_1\cup H_2$ is not homogeneous for $\mathcal{F}$. Therefore, there is a set $S\in\mathcal{F}$ which contains some, but not all, of $H_1\cup H_2$. Without loss of generality, we have $S\cap H_1=H_1$ and $S\cap H_2=\emptyset$ since $H_1$ and $H_2$ are homogeneous for $\mathcal{F}$. Now, pick $x\in H_1$ and $y\in H_2$ arbitrarily and define
\[T:=(S\setminus\{x\})\cup\{y\}.\]
Clearly $T$ cannot be in $\mathcal{F}$ since $T\cap H_1=H_1\setminus\{x\}$ and $H_1$ is homogeneous for $\mathcal{F}$. Since $\mathcal{F}$ is saturated, there must exist sets $A_1,\dots,A_k\in\mathcal{F}$ and $t\in\{0,\dots,k\}$ such that 
\[A_1\subsetneq \dots\subsetneq A_t\subsetneq T\subsetneq A_{t+1} \subsetneq \dots \subsetneq A_k.\]
Since $H_1$ and $H_2$ are homogeneous for $\mathcal{F}$, and neither $H_1$ nor $H_2$ is contained in $T$, we get that $A_t\subsetneq T\setminus(H_1\cup H_2) \subseteq S$. Similarly, $A_{t+1}\supsetneq S$. However, this implies that $\{A_1,\dots,A_k\}\cup\{S\}$ is a $(k+1)$-chain in $\mathcal{F}$, a contradiction. 
\end{proof}

By Lemma~\ref{block}, if $\mathcal{F}$ is a saturated $k$-Sperner system for which there exists a homogeneous set, then the homogeneous set must be unique. Throughout the paper, it will be useful to distinguish the elements of $\mathcal{F}$ which contain the homogeneous set from those that do not. 

\begin{defn}
Let $\mathcal{F}\subseteq\mathcal{P}(X)$ be a saturated $k$-Sperner system and let $H$ be homogeneous for $\mathcal{F}$.  We say that a set $S\in \mathcal{F}$ is \emph{large} if $H\subseteq S$ or \emph{small} if $S\cap H=\emptyset$. Let $\lrg{F}$ and $\sml{F}$ denote the collection of large and small sets of $\mathcal{F}$, respectively. Thus, $\mathcal{F}=\sml{F}\cup\lrg{F}$. 
\end{defn}

\begin{lem}
\label{thisorthat}
Let $\mathcal{A}\subseteq\mathcal{P}(X)$ be a saturated antichain with homogeneous set $H$. Then every set $S\in \mathcal{P}(X)\setminus\mathcal{A}$ either contains a set in $\sml{A}$ or is contained in a set of $\lrg{A}$. 
\end{lem}

\begin{proof}
Suppose, to the contrary, that $S\in \mathcal{P}(X)\setminus\mathcal{A}$ does not contain a set of $\sml{A}$ and is not contained in a set of $\lrg{A}$. Since $\mathcal{A}$ is saturated, we get that either 
\begin{enumerate}[(a)]
\item\label{bigcont} there exists $A\in \lrg{A}$ such that $A\subsetneq S$, or
\item\label{smallcont} there exists $B\in \sml{A}$ such that $S\subsetneq B$.
\end{enumerate}
Suppose that (\ref{bigcont}) holds. Let $y\in S\setminus A$ and $x\in H$ and define $T:=(A\setminus\{x\})\cup\{y\}$. Since $H$ is homogeneous for $\mathcal{A}$ and $T\cap H=H\setminus\{x\}$, we must have $T\notin\mathcal{A}$. Also, since $H$ is homogeneous for $\mathcal{A}$, any set $T'\in\mathcal{A}$ containing $T$ would have to contain $T\cup\{x\}\supsetneq A$. Therefore, since $\mathcal{A}$ is an antichain, no such set $T'$ can exist. Thus, there is a set $T''\in\mathcal{A}$ such that $T''\subsetneq T\subseteq S$. Since $H$ is homogeneous for $\mathcal{A}$ and $T\cap H\neq H$, we get that $T''\in\sml{A}$, contradicting our assumption on $S$.

Note that we are also done in the case that (\ref{smallcont}) holds by considering the saturated antichain $\{X\setminus A: A\in \mathcal{A}\}$ and applying the argument of the previous paragraph.
\end{proof}

\subsection{Constructing and Decomposing Saturated \texorpdfstring{$\boldsymbol{k}$}{k}-Sperner Systems}

There is a natural way to partition a $k$-Sperner system $\mathcal{F}\subseteq\mathcal{P}(X)$ into a sequence of $k$ pairwise disjoint antichains. Specifically, for $0\leq i\leq k-1$, let $\mathcal{A}_i$ be the collection of all minimal elements of $\mathcal{F} \setminus \left(\bigcup_{j<i}\mathcal{A}_j\right)$ under inclusion. We say that $\left(\mathcal{A}_i\right)_{i=0}^{k-1}$ is the \emph{canonical decomposition} of $\mathcal{F}$ into antichains. 

In this section we provide conditions under which a sequence of $k$ pairwise disjoint saturated antichains can be united to obtain a saturated $k$-Sperner system. Later we will prove a partial converse: if $\mathcal{F}\subseteq\mathcal{P}(X)$ is a saturated $k$-Sperner system with a homogeneous set, then every antichain of the canonical decomposition of $\mathcal{F}$ is saturated. We also provide an example which shows that this is not necessarily the case if we remove the condition that $\mathcal{F}$ has a homogeneous set. 

\begin{defn}
We say that a sequence $(\mathcal{D}_i)_{i=0}^{t}$ of subsets of $\mathcal{P}(X)$ is \emph{layered} if, for $1\leq i\leq t$, every $D\in\mathcal{D}_i$ strictly contains some $D'\in\mathcal{D}_{i-1}$ as a subset.
\end{defn}
Note that the canonical decomposition of any set system is layered.
\begin{lem}
\label{up}
If $(\mathcal{A}_i)_{i=0}^{t}$ is a layered sequence of pairwise disjoint saturated antichains, then every $A \in \mathcal{A}_i$ is strictly contained in some $B \in \mathcal{A}_{i+1}$
\end{lem}

\begin{proof}
Let $A \in \mathcal{A}_i$. Since $\mathcal{A}_{i+1}$ is a saturated antichain disjoint from $\mathcal{A}_{i}$, there exists some $B \in \mathcal{A}_{i+1}$ such that either $B \subsetneq A$ or $A \subsetneq B$. In the latter case we are done, so suppose $B \subsetneq A$. Since $(\mathcal{A}_i)_{i=0}^{t}$ is layered, there exists some $A' \in \mathcal{A}_i$ such that $A' \subsetneq B$. Hence we have $A' \subsetneq B \subsetneq A$, contradicting the fact that $\mathcal{A}_i$ is an antichain and completing the proof. 
\end{proof}

\begin{lem}
\label{cont}
If $(\mathcal{A}_i)_{i=0}^{k-1}$ is a layered sequence of pairwise disjoint saturated antichains in $\mathcal{P}(X)$, then $\mathcal{F}:=\bigcup_{i=0}^{k-1}\mathcal{A}_i$
is a saturated $k$-Sperner system. 
\end{lem}

\begin{proof}
Clearly, $\mathcal{F}$ is a $k$-Sperner system since $\mathcal{A}_0,\dots,\mathcal{A}_{k-1}$ are antichains. Let $S \in \mathcal{P}(X) \setminus \mathcal{F}$ be arbitrary and define $t = \max\{i : S \supsetneq A \text{ for some } A \in \mathcal{A}_i\}$. If $t\geq0$, then $S$ strictly contains some set $A_t \in \mathcal{A}_t$. As $(\mathcal{A}_i)_{i=0}^{k-1}$ is layered, for $0 \le i \le t-1$, there exist sets $A_i \in \mathcal{A}_i$ such that
$$A_0 \subsetneq \dots \subsetneq A_t \subsetneq S.$$
Now, if $t\geq k-2$, then since $\mathcal{A}_{t+1}$ is a saturated antichain and $S$ does not contain a set of $\mathcal{A}_{t+1}$, there must exist $A_{t+1}\in\mathcal{A}_{t+1}$ such that $S\subsetneq A_{t+1}$. By Lemma~\ref{up}, we see that for $t+2 \le i \le k-1$ there exists $A_i \in \mathcal{A}_i$ such that
$$S \subsetneq A_{t+1}  \subsetneq \dots \subsetneq A_{k-1}.$$
Thus $\{A_0,\dots,A_{k-1}\}\cup\{S\}$ is a $(k+1)$-chain, as desired. 
\end{proof}

In Lemma~\ref{cont}, we require the sequence $(\mathcal{A}_i)_{i=0}^{k-1}$ of saturated antichains to be layered. As it turns out, if each antichain $\mathcal{A}_i$ has a homogeneous set, then $(\mathcal{A}_i)_{i=0}^{k-1}$ is layered if and only if $\left(\sml{A}_i\right)_{i=0}^{k-1}$ is layered. 

\begin{lem}
\label{layer}
Let $(\mathcal{A}_i)_{i=0}^{k-1}$ be a sequence of pairwise disjoint saturated antichains in $\mathcal{P}(X)$, each of which has a homogeneous set. Then $(\mathcal{A}_i)_{i=0}^{k-1}$ is layered if and only if $\left(\sml{A}_i\right)_{i=0}^{k-1}$ is layered.
\end{lem}

\begin{proof}
Suppose that $(\mathcal{A}_i)_{i=0}^{k-1}$ is layered and, for some $i\geq0$, let $A\in \sml{A}_{i+1}$ be arbitrary. We show that $A$ contains a set of $\sml{A}_i$. Otherwise, since $(\mathcal{A}_i)_{i=0}^{k-1}$ is layered, we get that there is some $B\in\lrg{A}_i$ such that $B\subsetneq A$. Therefore, since $\mathcal{A}_i$ is an antichain, $A$ cannot be contained in an element of $\lrg{A}_i$. By Lemma~\ref{thisorthat} and the fact that $\mathcal{A}_i$ and $\mathcal{A}_{i+1}$ are disjoint, we get that $A$ contains a set of $\sml{A}_i$, as desired.

Now, suppose that $\left(\sml{A}_i\right)_{i=0}^{k-1}$ is layered. Given $i\geq0$ and $S\in \lrg{A}_{i+1}$, we show that $S$ contains a set of $\mathcal{A}_i$, which will complete the proof. If not, then since $\mathcal{A}_i$ is saturated and disjoint from $\mathcal{A}_{i+1}$, there must exist $T\in \mathcal{A}_i$ such that $S\subsetneq T$. Since $\mathcal{A}_{i+1}$ is an antichain, $S$ cannot be strictly contained in a set of $\lrg{A}_{i+1}$, and so neither can $T$. Therefore, by Lemma~\ref{thisorthat}, there is a set $A\in\sml{A}_{i+1}$ contained in $T$. However, since $\left(\sml{A}_i\right)_{i=0}^{k-1}$ is layered, there exists $A'\in\sml{A}_i$ such that $A'\subsetneq A$. But then, $A'\subsetneq T$, which contradicts the assumption that $\mathcal{A}_i$ is an antichain. The result follows. 
\end{proof}

It is natural to wonder whether a converse to Lemma~\ref{cont} is true. That is: \emph{if $\mathcal{F}$ is a saturated $k$-Sperner system, can we decompose $\mathcal{F}$ into a layered sequence of $k$ pairwise disjoint saturated antichains?} The following example shows that this is not always the case. 

\begin{ex}
\label{Cambridge}
Let $X := \{x_1,x_2,x_3\}$, $Y := \{y_1,y_2,y_3\}$ and $Z := X \cup Y$. We define 
\[\mathcal{B}_0:= \{\{x_i,x_j\}: i \not= j \} \cup \{\{x_i,y_i\}: i\in\{1,2,3\}\}\cup \{\{x_k,y_i,y_j\} : i,j,k \text{ distinct }\}\cup \{Y\}, \] 
\[\mathcal{B}_1:=\{X,\{x_1,x_2,y_1\},\{x_1,x_3,y_3\},\{x_2,x_3,y_2\},\{x_1,y_1,y_3\},\{x_2,y_1,y_2\},\{x_3,y_2,y_3\},\]
$$ \{x_1,x_2,y_2,y_3\}, \{x_1,x_3,y_1,y_2\}, \{x_2,x_3,y_1,y_3\}\}.$$ 
Then $\left(\mathcal{B}_i\right)_{i=0}^1$ is a layered sequence of disjoint antichains. In fact, $\left(\mathcal{B}_i\right)_{i=0}^1$ is the canonical decomposition of $\mathcal{F} := \mathcal{B}_0 \cup \mathcal{B}_1$. Clearly $\mathcal{B}_1$ is not saturated as $\mathcal{B}_1\cup\{Y\}$ is an antichain. We claim that $\mathcal{F}$ is a saturated $2$-Sperner system.

Consider any $S \in \mathcal{P}(Z)\setminus \mathcal{F}$. We will show that $\mathcal{F}\cup\{S\}$ contains a $3$-chain. It is easy to check that every element of $\mathcal{B}_0\setminus \{Y\}$ is contained in a set of $\mathcal{B}_1$. Hence if $S$ is contained in some set $B \in \mathcal{B}_0\setminus \{Y\}$, then $\mathcal{F}\cup\{S\}$ contains a $3$-chain. In particular, this completes the proof when $|S| \in \{0,1,2\}$. Similarly, since $\left(\mathcal{B}_i\right)_{i=0}^1$ is layered, if $S$ contains some set $B \in \mathcal{B}_1$, then $\mathcal{F}\cup\{S\}$ contains a $3$-chain. Therefore, we are done if $|S| \in \{4,5,6\}$. 

It remains to consider the case that $|S| = 3$. Since $X,Y\in \mathcal{F}$, we must have $|S \cap Y| = 2$, or $|S \cap X| = 2$. If $|S \cap Y| = 2$, we have $S \in \{\{x_1, y_1, y_2\}, \{x_2,y_2,y_3\}, \{x_3, y_1, y_3\}\}$. This implies that $S$ is contained in a set $B \in \mathcal{B}_1$ and contains a set $B' \in \mathcal{B}_0 \cap \mathcal{P}(X)$. If $|S\cap X|=2$, then $S$ contains some set $\{x_i,x_j\}\in\mathcal{B}_0$. Also, it is easily verified that $S$ is contained in a set of $\mathcal{B}_1$. Thus, $\mathcal{F}$ is a saturated 2-Sperner system.
\end{ex}

However, for saturated $k$-Sperner systems with a homogeneous set, the converse to Lemma~\ref{cont} does hold; we can partition $\mathcal{F}$ into a layered sequence of $k$ pairwise disjoint saturated antichains. 

\begin{lem}\label{satur}
Let $\mathcal{F} \in \mathcal{P}(X)$ be a saturated $k$-Sperner system with homogeneous set $H$ and canonical decomposition $\left(\mathcal{A}_i\right)_{i=0}^{k-1}$. Then $\mathcal{A}_i$ is saturated for all $i$.
\end{lem}

\begin{proof}
Fix $i$ and let $S\in \mathcal{P}(X)\setminus \mathcal{A}_i$. Let $x\in H$ and define
\[T:=(S\setminus H)\cup\{x\}.\]
Then $T\notin \mathcal{F}$ since $T\cap H=\{x\}$ and $H$ is homogeneous for $\mathcal{F}$. Therefore, there exists $\{A_0,\dots,A_{k-1}\}\subseteq\mathcal{F}$ and $t \in \{0,\ldots, k\}$ such that
\[A_0\subsetneq \dots\subsetneq A_{t-1}\subsetneq T\subsetneq A_{t}\subsetneq\dots\subsetneq A_{k-1}.\]
By definition of the canonical decomposition, we must have $A_j\in\mathcal{A}_j$ for all $j$. Also, since $H$ is homogeneous for $\mathcal{F}$ and $T\cap H\notin\{\emptyset,H\}$, we must have $A_{t-1}\subseteq T\setminus H \subseteq S$ and $A_{t}\supseteq T\cup H \supseteq S$. Therefore,
\[A_0\subsetneq \dots\subsetneq A_{t-1}\subseteq S\subseteq A_{t}\subsetneq\dots\subsetneq A_{k-1}.\]
Since $S\neq A_i$, we must have either $A_i\subsetneq S$ or $S\subsetneq A_i$ depending on whether or not $i<t$. Therefore, $\mathcal{A}_i$ is saturated for all $i$. 
\end{proof}

\section{Combining Saturated \texorpdfstring{$\boldsymbol{k}$}{k}-Sperner Systems}
\label{comb}

Our first goal in this section is to prove that, under certain conditions, a saturated $k_1$-Sperner system $\mathcal{F}_1\subseteq\mathcal{P}(X_1)$  and a saturated $k_2$-Sperner system $\mathcal{F}_2\subseteq\mathcal{P}(X_2)$ can be combined to yield a saturated $(k_1 +k_2-2)$-Sperner system in $\mathcal{P}(X_1\cup X_2)$. We apply this result to prove Theorem~\ref{epsilon}. Afterwards, we prove that  $\sat(k)=2^{k-1}$ for $k\leq 5$. We conclude the section with a proof of Theorem~\ref{submult}.

\begin{lem}
\label{comblem}
Let $X_1$ and $X_2$ be disjoint sets. For $i\in\{1,2\}$, let $\mathcal{F}_i\subseteq \mathcal{P}(X_i)$ be a saturated $k_i$-Sperner system which contains $\{\emptyset,X_i\}$ and let $H_i\subseteq X_i$ be homogeneous for $\mathcal{F}_i$. If $\mathcal{G}$ is the set system on $\mathcal{P}(X_1\cup X_2)$ defined by
\[\mathcal{G}:=\left\{A\cup B: A\in\sml{F}_1, B\in\sml{F}_2\right\}\cup\left\{S\cup T: S\in\lrg{F}_1, T\in\lrg{F}_2\right\},\]
then $\mathcal{G}$ is a saturated $(k_1+k_2-2)$-Sperner system which contains $\{\emptyset,X_1\cup X_2\}$ and $H_1\cup H_2$ is homogeneous for $\mathcal{G}$. 
\end{lem}

\begin{proof}
It is clear that $\mathcal{G}$ contains $\{\emptyset,X_1\cup X_2\}$ and that $H_1\cup H_2$ is homogeneous for $\mathcal{G}$. We show that $\mathcal{G}$ is a saturated $(k_1+k_2-2)$-Sperner system. 

First, let us show that $\mathcal{G}$ does not contain a chain of length $k_1+k_2-1$. Suppose that $\{A_1,\dots,A_{r}\}$ is an $r$-chain in $\mathcal{G}$. We can assume that $A_1=\emptyset$ and $A_{r}=X_1\cup X_2$. Define
\[I_1:=\{i: A_i\cap X_1\subsetneq A_{i+1}\cap X_1\},\text{ and}\]
\[I_2:=\{i: A_i\cap X_2\subsetneq A_{i+1}\cap X_2\}.\]
Clearly, $I_1\cup I_2=\{1,\dots,r-1\}$. Also, for $i\in\{1,2\}$, since $\mathcal{F}_i$ is a $k_i$-Sperner system, we must have $|I_i|\leq k_i-1$. Let $t$ be the maximum index such that $A_t\cap X_1\in \sml{F}_1$. Note that $t$ exists and is less than $r$ since $A_1=\emptyset$ and $A_r=X_1\cup X_2$. By construction of $\mathcal{G}$, $A_t\cap X_2$ is a small set for $\mathcal{F}_2$ and, for $i\in\{1,2\}$, $A_{t+1}\cap X_i$ is a large set for $\mathcal{F}_i$. This implies that $t\in I_1\cap I_2$ and so
\[r-1 = |I_1\cup I_2| = |I_1|+|I_2|-|I_1\cap I_2|\leq k_1+k_2-3\]
as required. 

Now, let $S\in\mathcal{P}(X_1\cup X_2)\setminus\mathcal{G}$. We show that $\mathcal{G}\cup\{S\}$ contains a $(k_1+k_2-1)$-chain. Fix $x_1\in H_1$ and $x_2\in H_2$ and define
\[T:=(S\setminus (H_1\cup H_2))\cup\{x_1,x_2\}.\]
For $i\in\{1,2\}$, let $T_i:= T\cap X_i$. Then $T_i\notin\mathcal{F}_i$ since $T_i\cap H_i =\{x_i\}$. Therefore, there exists $A_1^i,\dots,A_{k_i}^i\in\mathcal{F}_i$ and $t_i\in\{1,\dots,k_i-1\}$ such that
\[\emptyset = A_1^i\subsetneq \dots\subsetneq A_{t_i}^i\subsetneq T_i\subsetneq A_{t_i+1}^i \subsetneq \dots\subsetneq A_{k_i}^i=X_i\]
Note that $A_j^i\in \sml{F}_i$ for $j\leq t_i$ and $A_j^i\in \lrg{F}_i$ for $j\geq t_i+1$. This implies that $A_{t_1}^1\cup A_{t_2}^2\subsetneq S$ and $A_{t_1+1}^1\cup A_{t_2+1}^2\supsetneq S$. Therefore,
\[A_1^1\cup A_1^2\subsetneq A_1^1\cup A_2^2\subsetneq\dots \subsetneq A_1^1\cup A_{t_2}^2\subsetneq A_2^1\cup A_{t_2}^2 \subsetneq \dots \subsetneq A_{t_1}^1\cup A_{t_2}^2\subsetneq S\]
\[\subsetneq A_{t_1+1}^1\cup A_{t_2+1}^2\subsetneq A_{t_1+1}^1\cup A_{t_2+2}^2\subsetneq \dots\subsetneq A_{t_1+1}^1\cup A_{k_2}^2\subsetneq A_{t_1+2}^1\cup A_{k_2}^2\subsetneq\dots\subsetneq A_{k_1}^2 \cup A_{k_2}^2\]
and so $\mathcal{G}\cup\{S\}$ contains a $(k_1+k_2-1)$-chain. The result follows.
\end{proof}

\begin{rem}
\label{sizes}
If $\mathcal{F}_1$, $\mathcal{F}_2$ and $\mathcal{G}$ are as in Lemma~\ref{comblem}, then
\[|\mathcal{G}| = \left|\sml{F}_1\right|\left|\sml{F}_2\right| + \left|\lrg{F}_1\right|\left|\lrg{F}_2\right|.\]
\end{rem}

\subsection{Proof of Theorem~\ref{epsilon}}

We apply Lemma~\ref{comblem} to prove Theorem~\ref{epsilon}. The first part of the proof of Theorem~\ref{epsilon} is to exhibit an infinite family of saturated $6$-Sperner systems with cardinality $30<2^5$. 

\begin{prop}
\label{k=6}
For any set $X$ such that $|X|\geq 8$, there is a saturated $6$-Sperner system $\mathcal{F}\subseteq\mathcal{P}(X)$ with a homogeneous set such that $\left|\sml{F}\right|=\left|\lrg{F}\right|=15$. 
\end{prop}

\begin{proof}
Let $X$ be a set such that $|X|\geq8$. Let $x_1,x_2,y_1,y_2,w$ and $z$ be distinct elements of $X$ and define $H:=X\setminus \{x_1,x_2,y_1,y_2,w,z\}$. We apply Lemma~\ref{cont} to construct a saturated $6$-Sperner system $\mathcal{F}\subseteq\mathcal{P}(X)$ of order $30$. Naturally, we define $\mathcal{A}_0=\{\emptyset\}$ and $\mathcal{A}_5:=\{X\}$. Also, define
\[\mathcal{A}_1:=\{\{x_1\},\{x_2\},\{y_1\},\{w\}, H\cup\{y_2,z\}\},\text{ and}\]
\[\mathcal{A}_4:=\{X\setminus A: A\in\mathcal{A}_1\}.\]

It is easily observed that $\mathcal{A}_1$ and $\mathcal{A}_4$ are saturated antichains. We define $\mathcal{A}_2$ and $\mathcal{A}_3$ by first specifying their small sets. Define 
\[\sml{A}_2:=\{\{x_i,y_j\}: 1\leq i,j\leq 2\}\cup \{\{w,z\}\},\text{ and}\]
\[\sml{A}_3:=\{\{x_1,y_1,w\}, \{x_1,y_1,z\}, \{x_2,y_2,w\}, \{x_2,y_2,z\}\}.\]
Given any collection $\mathcal{B}\subseteq\mathcal{P}(X)$, a set $S\subseteq X$ is said to be \emph{stable} for $\mathcal{B}$ if $S$ does not contain an element of $\mathcal{B}$. For $i=2,3$, define $\lrg{A}_i$ to be the collection consisting of all maximal stable sets of $\sml{A}_i$ and let $\mathcal{A}_i:=\sml{A}_i\cup \lrg{A}_i$. Note that every element of $\lrg{A}_i$ contains $H$. It is clear that $\mathcal{A}_i$ is an antichain for $i=2,3$. Moreover, $\mathcal{A}_i$ is saturated since every set $A\in \mathcal{P}(X)$ either contains an element of $\sml{A}_i$ or is contained in an element of $\lrg{A}_i$.

One can easily verify that $\left(\sml{A}_i\right)_{i=0}^{5}$ is layered. Therefore, by Lemma~\ref{layer}, $\left(\mathcal{A}_i\right)_{i=0}^{5}$ is a layered sequence of pairwise disjoint saturated antichains. By Lemma~\ref{cont}, $\mathcal{F}:=\bigcup_{i=0}^5\mathcal{A}_i$ is a saturated 6-Sperner system.
Also,
\[\left|\sml{F}\right| = \sum_{i=0}^5 \left|\sml{A}_i\right| = (1+5+9+0)=15,\text{ and}\]
\[\left|\lrg{F}\right|=\sum_{i=0}^5 \left|\lrg{A}_i\right|= (0+9+5+1) = 15,\]
as desired. 
\end{proof}

We remark that the construction in Proposition~\ref{k=6} is similar to one which was used in~\cite{gerb} to prove that $\sat(k,k)\leq \frac{15}{16}2^{k-1}$ for every $k\geq6$. 

For the proof of Theorem~\ref{epsilon} we require that 
\begin{equation}\label{2sat}\sat(k)\leq 2\sat(k-1).\end{equation}
This was proved in~\cite{gerb} using the fact that if $\mathcal{F}\subseteq\mathcal{P}(X)$ is a saturated $(k-1)$-Sperner system and $y\notin X$, then $\mathcal{F}\cup\{A\cup\{y\}: A\in\mathcal{F}\}$
is a saturated $k$-Sperner system in $\mathcal{P}(X\cup\{y\})$. 

\begin{proof}[Proof of Theorem~\ref{epsilon}]
First, we prove that the result holds when $k$ is of the form $4j+2$ for some $j\geq1$. In this case, we repeatedly apply Lemma~\ref{comblem} and  Proposition~\ref{k=6} to obtain a saturated $k$-Sperner system $\mathcal{F}$ on an arbitrarily large ground set $X$ such that
\[\left|\sml{F}\right| + \left|\lrg{F}\right| = 15^j + 15^j = 2\cdot15^j.\]
Therefore, if $k=4j + 2$, then $\sat(k)\leq 2\cdot15^j$. 

For $k$ of the form $4j+2+s$ for $j\geq1$ and $1\leq s\leq 3$, apply (\ref{2sat}) to obtain $\sat(k)\leq 2^s\sat(4j+2)\leq 2^{s+1}\cdot15^j$. Thus, we are done by setting $\varepsilon$ slightly smaller than $\left(1 - \frac{\log_{2}\left(15\right)}{4}\right)$.
\end{proof}

\subsection{Bounding \texorpdfstring{$\boldsymbol{\sat(k)}$}{sat(k)} From Below}

One can easily deduce from the proof of Theorem~\ref{epsilon} that $\sat(k) < 2^{k-1}$ for all $k\geq 6$. For completeness, we prove that $\sat(k)=2^{k-1}$ for $k\leq 5$. 

\begin{prop}
\label{k<=5}
If $k\leq 5$, then $\sat(k)=2^{k-1}$. 
\end{prop}

\begin{proof}
Fix $k\leq 5$. The upper bound follows from Construction~\ref{powerset}, and so it suffices to prove that $\sat(k)\geq2^{k-1}$. Let $X$ be a set with $n:=|X|> 2^{2^{k-1}}$ and let $\mathcal{F}\subseteq\mathcal{P}(X)$ be a saturated $k$-Sperner system of minimum order. By Claim~\ref{bighom} and the fact that $|X|>2^{2^{k-1}}\geq2^{|\mathcal{F}|}$, there is a homogeneous set $H$ for $\mathcal{F}$. 

Let $\left(\mathcal{A}_i\right)_{i=0}^{k-1}$ be the canonical decomposition of $\mathcal{F}$. By Lemma~\ref{satur},  we get that $\mathcal{A}_i$ is a saturated antichain for each $i$. Also, since $\left(\mathcal{A}_i\right)_{i=0}^{k-1}$ is layered, by Lemma~\ref{up} we see that 
\begin{equation}\label{between}\text{every element of $\mathcal{A}_i$ has cardinality between $i$ and $n-k+i+1$.} \end{equation}
Our goal is to to show that for $k\leq5$, every saturated antichain $\mathcal{A}_i$ which satisfies (\ref{between}) must contain at least $\binom{k-1}{i}$ elements. Clearly this is enough to complete the proof of the proposition. Note that it suffices to prove this for $i<\frac{k}{2}$ since $\{X\setminus A: A\in\mathcal{A}_i\}$ is a saturated antichain in which every set has size between $k-i-1$ and $n-i$. Since $k\leq5$, this means that we need only check the cases $i=0,1,2$. In the case $i=0$, we obtain $|\mathcal{A}_0|\geq 1 = \binom{k-1}{0}$ trivially. 

Next, consider the case $i=1$. Let $A$ be the largest set in $\mathcal{A}_1$ such that $H\subseteq A$.  Then, by (\ref{between}), we must have $|A|\leq n-k+2$ and so $|X\setminus A|\geq k-2$. Fix an element $x$ of $H$ and, for each $y\in X\setminus A$, define $A_y:=(A\setminus\{x\})\cup\{y\}$. Since $\mathcal{A}_1$ is saturated, $H$ is homogeneous for $\mathcal{F}$,  and $A$ is the largest set in $\mathcal{A}_1$ containing $H$, there must be a set $B_y\in\mathcal{A}_1$ such that $B_y\subsetneq A_y$. However, since $\mathcal{A}_1$ is an antichain, $B_y\nsubseteq A$, and so $B_y\setminus A = \{y\}$. In particular, $B_y\neq B_{y'}$ for $y\neq y'$. Therefore, $|\mathcal{A}_1|\geq |\{A\}\cup\{B_y:y\in X\setminus A\}| \geq 1+|X\setminus A| \geq k-1 = \binom{k-1}{1}$, as desired. 

Thus, we are finished except for the case $i=2$ and $k=5$. Suppose to the contrary that $|\mathcal{A}_2|<\binom{4}{2}=6$. We begin by proving the following claim.

\begin{claim}
\label{conty}
For every vertex $y\in X\setminus H$, there is a set $S_y\in\lrg{A}_2$ containing $y$.
\end{claim}

\begin{proof}
Let $x\in H$ be arbitrary and consider the set $T:=\{x,y\}$. Then $T$ is not contained in $\mathcal{A}_2$ since $H$ is homogeneous for $\mathcal{F}$. Also, no strict subset of $T$ is in $\mathcal{A}_2$ by (\ref{between}). Since $\mathcal{A}_2$ is saturated, there must be some $S_y\in\lrg{A}_2$ containing $T$, which completes the proof.
\end{proof}

Let us argue that $\left|\lrg{A}_2\right|\geq 3$. By (\ref{between}), each set $A\in\lrg{A}_2$ has at most $n-2$ elements. So, by Claim~\ref{conty}, if $\left|\lrg{A}_2\right|<3$, then it must be the case that $\lrg{A}_2 = \{A_1,A_2\}$ where $A_1\cup A_2=X$. Therefore, since each of $|A_1|$ and $|A_2|$ is at most $n-2$, we can pick $\{w_1,w_2\}\subseteq A_1\setminus A_2$ and $\{z_1,z_2\}\subseteq A_2\setminus A_1$. Given $x\in H$ and $1\leq i,j\leq 2$, we have that $\{x,w_i,z_j\}\notin \mathcal{A}_2$ since $H$ is homogeneous for $\mathcal{F}$. Note that $\{x,w_i,z_j\}$ is not contained in either $A_1$ or $A_2$, and so by Lemma~\ref{thisorthat} and (\ref{between}) we must have $\{w_i,z_j\}\in\mathcal{A}_2$. However, this implies that $|\mathcal{A}_2|\geq |\{\{w_i,z_j\}:1\leq i,j\leq 2\}\cup\{A_1,A_2\}|=6$, a contradiction.

So, we get that $\left|\lrg{A}_2\right|\geq 3$. Note that $\{X\setminus A: A\in\mathcal{A}_2\}$ is also a saturated antichain in which every set has cardinality between $2$ and $n-2$. Thus, we can apply the argument of the previous paragraph to obtain $\left|\sml{A}_2\right|\geq 3$. Therefore, $\left|\mathcal{A}_2\right|=\left|\sml{A}_2\right| + \left|\lrg{A}_2\right| \geq 6$, which completes the proof.
\end{proof}

It is possible that a similar approach may prove fruitful for improving the lower bound on $\sat(k)$ from $2^{k/2-1}$ to $2^{(1+o(1))ck}$ for some $c>1/2$. That is, one may first decompose a saturated $k$-Sperner system $\mathcal{F}\subseteq\mathcal{P}(X)$ of minimum size into its canonical decomposition $\left(\mathcal{A}_i\right)_{i=0}^{k-1}$ and then bound the size of $|\mathcal{A}_i|$ for each $i$ individually. Since there are only $k$ antichains in the decomposition and the bound on $|\mathcal{F}|$ that we are aiming for is exponential in $k$, one could obtain a fairly tight lower bound on $\sat(k)$ by focusing on a single antichain of the decomposition. Setting $i=\left\lfloor\frac{k}{2}\right\rfloor$ in (\ref{between}), we see that it would be sufficient to prove that there exists $c>1/2$ such that every saturated antichain $\mathcal{A}$ with a homogeneous set such that every element of $\mathcal{A}$ has cardinality between $\left\lfloor\frac{k}{2}\right\rfloor$ and $n-\left\lceil\frac{k}{2}\right\rceil+1$ must satisfy $|\mathcal{A}|\geq 2^{(1+o(1))ck}$. The problem of determining whether such a $c$ exists is interesting in its own right. 

\subsection{Asymptotic Behaviour of \texorpdfstring{$\boldsymbol{\sat(k)}$}{sat(k)}}
\label{asymp}

To prove Theorem~\ref{submult}, we require the following fact, which is proved in~\cite{gerb}.

\begin{lem}[Gerbner et al.~\cite{gerb}]
\label{emptyfull}
For any $n\geq k\geq1$ and set $X$ with $|X|=n$ there is a saturated $k$-Sperner system $\mathcal{F}\subseteq\mathcal{P}(X)$ such that $|\mathcal{F}|=\sat(n,k)$ and $\{\emptyset,X\}\subseteq \mathcal{F}$. 
\end{lem}

\begin{proof}
Let $\mathcal{F}\subseteq\mathcal{P}(X)$ be a saturated $k$-Sperner system such that $|\mathcal{F}|=\sat(n,k)$. We let $\left(\mathcal{A}_i\right)_{i=0}^{k-1}$ denote the canonical decomposition of $\mathcal{F}$ and define
\[\mathcal{F}':=\left(\mathcal{F} \setminus \left(\mathcal{A}_0\cup\mathcal{A}_{k-1}\right)\right)\cup\{\emptyset,X\}.\]
It is clear that $\mathcal{F}'\subseteq \mathcal{P}(X)$ is a saturated $k$-Sperner system and $|\mathcal{F}'|\leq |\mathcal{F}|=\sat(n,k)$, which proves the result. 
\end{proof}

\begin{proof}[Proof of Theorem~\ref{submult}]
We show that, for all $k,\ell$,
\begin{equation}\label{4sat}\sat(k+\ell)\leq 4\sat(k)\sat(\ell).\end{equation}
Letting $f(k) := 4\sat(k)$, we see that (\ref{4sat}) implies that $f(k+\ell)\leq f(k)f(\ell)$ for every $k,\ell$. It follows by Fekete's Lemma that $f(k)^{1/k}$ converges, and so $\sat(k)^{1/k}$ converges as well.

For $n>2^{2^{k+\ell-2}}$, let $X$ and $Y$ be disjoint sets of size $n$ and let $\mathcal{F}_k\subseteq\mathcal{P}(X)$ and $\mathcal{F}_\ell\subseteq\mathcal{P}(Y)$ be saturated $k$-Sperner and $\ell$-Sperner systems of cardinalities $\sat(k)$ and $\sat(\ell)$, respectively. By Claim~\ref{bighom}, we can assume that $\mathcal{F}_k$ and $\mathcal{F}_\ell$ have homogeneous sets and, by Lemma~\ref{emptyfull}, we can assume that $\{\emptyset,X\}\subseteq\mathcal{F}_k$ and $\{\emptyset,Y\}\subseteq\mathcal{F}_\ell$. We apply Lemma~\ref{comblem} and Remark~\ref{sizes} to obtain a saturated $(k+\ell-2)$-Sperner system $\mathcal{G}\subseteq \mathcal{P}(X \cup Y)$ of order at most $|\mathcal{F}_k||\mathcal{F}_\ell|=\sat(k)\sat(\ell)$. Therefore, by (\ref{2sat}), we have
\[\sat(k+\ell)\leq 4\sat(k+\ell-2)\leq 4|\mathcal{G}| \le 4\sat(k)\sat(\ell)\]
as required. 
\end{proof}

\section{Oversaturated \texorpdfstring{$\boldsymbol{k}$}{k}-Sperner Systems}\label{wsat}

In this section we construct oversaturated $k$-Sperner systems of small order. We first state a lemma, from which Theorem~\ref{wsatb} follows, and then prove the lemma itself. 

\begin{lem}\label{fun}
Given $k\geq 1$, let $X$ be a set of cardinality $k^2 + k$. Then for all $t$ such that $1 \le t \le k^2+k$ there exist non-empty collections $\mathcal{F}_t$, $\mathcal{G}_t \subseteq \mathcal{P}(X)$ that have the following properties:
\begin{enumerate}[(a)]
\item For every $F \in \mathcal{F}_t$ and $G \in \mathcal{G}_t$, $|F| + |G| \ge k$,\label{sumk} 
\item $|\mathcal{F}_t| + |\mathcal{G}_t| = O\left(k^2 2^{k/2}\right)$,\label{nottoomany}
\item For every $S \subseteq X$ such that $|S|=t$, there exists some $F \in \mathcal{F}_t$ and some $G \in \mathcal{G}_t$ such that $F \subsetneq S$ and $G \cap S = \emptyset$. \label{contained}
\end{enumerate}
\end{lem}

We apply Lemma~\ref{fun} to prove Theorem~\ref{wsatb}.

\begin{proof}[Proof of Theorem~\ref{wsatb}]
First, let $X$ be a set of cardinality $k^2+k$. For $t\in\{1,\dots,k^2+k\}$, let $\mathcal{F}_t$ and $\mathcal{G}_t$ be as in Lemma~\ref{fun}. For each $F \in \mathcal{F}_t\cup\mathcal{G}_t$, choose $F_1,\dots, F_i\in\mathcal{P}(X)$ such that
\[F_1\subsetneq \dots\subsetneq F_i\subsetneq F\]
where $i:=\min\{k-1, |F|\}$. We let $\mathcal{C}_F:=F \cup\{F_1,\dots,F_i\}$ and define
$$\mathcal{G} := \bigcup_{1 \le t \le k^2 + k} \left( \{T:T \in \mathcal{C}_F \text{ for some } F \in \mathcal{F}_t\} \cup \{X\setminus T :T \in \mathcal{C}_G \text{ for some } G \in \mathcal{G}_t\} \right).$$ \\
For each $t\leq k^2+k$ and $F \in \mathcal{F}_t\cup\mathcal{G}_t$, we have $|\mathcal{C}_F|\leq k$. Thus, by Property (\ref{nottoomany}) of Lemma~\ref{fun}, 
$$|\mathcal{G}| \le \sum_{t=1}^{k^2+k}k(|\mathcal{F}_t|+|\mathcal{G}_t|) = O\left(k^52^{k/2}\right).$$

We will now show that for any $S \in \mathcal{P}(X) \setminus \mathcal{G}$ there is a $(k+1)$-chain in $\mathcal{G}\cup\{S\}$ containing $S$, which will imply that $\mathcal{G}$ is an oversaturated $k$-Sperner system. Let $S \subseteq X$ and define $t:=|S|$. By Property (\ref{contained}) of Lemma~\ref{fun}, there exists $F \in \mathcal{F}_t$ such that $F \subsetneq S$ and $G \in \mathcal{G}_t$ such that $G \cap S = \emptyset$. This implies that $S \subsetneq X\setminus G$. By Property (\ref{sumk}) of Lemma~\ref{fun} we get that 
\[\mathcal{C}_F\cup\{X\setminus T: T\in\mathcal{C}_G\}\cup\{S\}\]
contains a $(k+1)$-chain in $\mathcal{G}\cup\{S\}$ containing $S$.

Now, suppose that $|X|> k^2+k$. Let $Y\subseteq X$ such that $|Y|=k^2+k$ and define $H:=X\setminus Y$. As above, let $\mathcal{G}\subseteq\mathcal{P}(Y)$ be an oversaturated $k$-Sperner system of cardinality at most $O\left(k^52^{k/2}\right)$. Define $\mathcal{G'} \subseteq \mathcal{P}(X)$ as follows:
$$\mathcal{G'} := \{T : T \in \mathcal{G}\} \cup \{T\cup H: T \in \mathcal{G}\}.$$
Consider any set $S \in \mathcal{P}(X)\setminus \mathcal{G'}$. Let $S' = S \cap Y$. We have, by definition of $\mathcal{G}$, that there is a $(k+1)$-chain $\mathcal{C}$ in $\mathcal{G}\cup\{S'\}$ containing $S'$. Adding $H$ to every superset of $S'$ in $\mathcal{C}$ and replacing $S'$ by $S$ in $\mathcal{C}$ gives us a $(k+1)$-chain in $\mathcal{G}'\cup\{S\}$ containing $S$. The result follows.
\end{proof}

To prove Lemma~\ref{fun}, we use a probabilistic approach.

\begin{proof}[Proof of Lemma~\ref{fun}]
Throughout the proof, we assume that $k$ is sufficiently large and let $X$ be a set of cardinality $k^2+k$. Let $1\leq t\leq k^2+k$ be given. We can assume that $t\leq \frac{k^2+k}{2}$ since, otherwise, we can simply define $\mathcal{F}_t:=\mathcal{G}_{k^2+k-t}$ and $\mathcal{G}_t:=\mathcal{F}_{k^2+k-t}$. We divide the proof into two cases depending on the size of $t$.

\begin{case}
$t\leq \frac{k^2+k}{8}$.
\end{case}

We define $\mathcal{F}_t:=\{\emptyset\}$ and let $\mathcal{G}_t$ be a uniformly random collection of $2^{k/2}$ subsets of $X$, each of cardinality $k$. Given $S \subseteq X$ of cardinality $t$, the probability that $S$ is not disjoint from any set of $\mathcal{G}_t$ is 
\[\left(1-\prod_{i=0}^{k-1}\left(\frac{k^2 + k -t-i}{k^2+k-i}\right)\right)^{2^{k/2}}\leq \left(1-\left(\frac{k^2 -t}{k^2}\right)^k\right)^{2^{k/2}} \leq \left(1-\left(\frac{7}{8}-\frac{1}{8k}\right)^k\right)^{2^{k/2}}\]
\[\leq e^{-\left(\frac{7}{8} -\frac{1}{8k}\right)^k2^{k/2}}< e^{-(1.1)^k}.\]
Therefore, the expected number of subsets of $X$ of cardinality $t$ which are not disjoint from any set of $\mathcal{G}_t$ is at most $\binom{k^2+k}{t}e^{-(1.1)^k}$,
which is less than $1$. Thus, with non-zero probability, every  $S\subseteq X$ of cardinality $t$ is disjoint from some set in $\mathcal{G}_t$. 

\begin{case}
$\frac{k^2+k}{8} < t \leq\frac{k^2+k}{2}$.
\end{case}

Define $p:=\frac{t}{k^2+k}$ and let $a$ be the rational number such that $ak = \left\lfloor\frac{-k\log{\sqrt{2}}}{\log(p)} + 1\right\rfloor$. Then, since $\frac{1}{8}\leq p\leq \frac{1}{2}$, we have
\begin{equation}\label{bounds}1/6\leq a\leq 1/2 + 1/k < 4/7.\end{equation}

Now, let $\mathcal{F}_t$ be a collection of $\left\lceil8e^8k^22^{k/2}\right\rceil$ subsets of $X$, each of cardinality $ak$, chosen uniformly at random with replacement. Similarly, let $\mathcal{G}_t$ be a collection of $\left\lceil e^2k^22^{k/2}\right\rceil$ subsets of $X$, each of cardinality $(1-a)k$, chosen uniformly at random with replacement. We show that, with non-zero probability, every $S \subseteq X$ of size $t$ contains a set of $\mathcal{F}_t$ and is disjoint from a set of $\mathcal{G}_t$. 

Given $S\subseteq X$ of size $t=p(k^2+k)$, the probability that $S$ does not contain a set of $\mathcal{F}_t$ is at most
\[\left(1-\prod_{i=0}^{ak-1}\left(\frac{p(k^2+k)-i}{k^2+k-i}\right)\right)^{|\mathcal{F}_t|}\leq \left(1-\left(\frac{p(k^2+k) -k}{k^2}\right)^{ak}\right)^{|\mathcal{F}_t|}\]
\begin{equation}\label{pk}=\left(1-\left(1 - \frac{1-p}{pk}\right)^{ak}p^{ak}\right)^{|\mathcal{F}_t|}.\end{equation}
Observe that $\left(1 - \frac{1-p}{pk}\right) \geq e^{-\frac{2(1-p)}{pk}}$ for large enough $k$. So, $\left(1 - \frac{1-p}{pk}\right)^{ak}\geq e^{\frac{-2a(1-p)}{p}}$ which is at least $e^{-8}$ since $a<4/7$ and $p\geq 1/8$. Thus, the expression in (\ref{pk}) is at most
\[\left(1-e^{-8}p^{ak}\right)^{|\mathcal{F}_t|} \leq e^{-e^{-8}p^{ak}|\mathcal{F}_t|}\leq e^{-e^{-8}p^{ak}\left(8e^8k^22^{k/2}\right)} = e^{-p^{ak}8k^22^{k/2}}.\]
Using our choice of $a$ and the fact that $p\geq 1/8$, we can bound the exponent by
\[p^{ak}8k^22^{k/2} \geq p^{\left(-\frac{\log{\sqrt{2}}}{\log(p)} + \frac{1}{k}\right)k}8k^22^{k/2}= p8k^2\geq k^2.\]
Therefore, the expected number of subsets of $X$ of size $t$ which do not contain a set of $\mathcal{F}_t$ is at most
\[\binom{k^2+k}{t}e^{-k^2} < 2^{k^2+k}e^{-k^2}\]
which is less than $1$. Thus, with positive probability, every subset of $X$ of cardinality $t$ contains a set of $\mathcal{F}_t$. 

The proof that, with positive probability, every set of cardinality $t$ is disjoint from a set of $\mathcal{G}_t$ is similar; we sketch the details. First, let us note that
\begin{equation}\label{abound}a\geq \frac{-\log{\sqrt{2}}}{\log(p)} \geq 1 + \frac{\log{\sqrt{2}}}{\log(1-p)}\end{equation}
since $p\leq 1/2$. For a fixed set $S\subseteq X$ of size $t=p(k^2+k)$, the probability that $S$ is not disjoint from any set of $\mathcal{G}_t$ is at most
\[\left(1-\prod_{i=0}^{(1-a)k-1}\left(\frac{(1-p)(k^2+k)-i}{k^2+k-i}\right)\right)^{|\mathcal{G}_t|}\leq \left(1-\left(\frac{(1-p)(k^2+k) -k}{k^2}\right)^{(1-a)k}\right)^{|\mathcal{G}_t|}\]
\begin{equation}\label{Gtbound}=\left(1-\left(1 - \frac{p}{(1-p)k}\right)^{(1-a)k}(1-p)^{(1-a)k}\right)^{|\mathcal{G}_t|}\end{equation}
Now, $\left(1 - \frac{p}{(1-p)k}\right)\geq e^{\frac{-2p}{(1-p)k}}$ for large enough $k$. So,  $\left(1 - \frac{p}{(1-p)k}\right)^{(1-a)k}\geq e^{\frac{-2(1-a)p}{(1-p)}}$, which is at least $e^{-2}$ since $a\geq 1/6$ and $\frac{1}{8}\leq p\leq \frac{1}{2}$. Therefore, the expression in (\ref{Gtbound}) is at most
\[\left(1-e^{-2}(1-p)^{(1-a)k}\right)^{|\mathcal{G}_t|} \leq e^{-e^{-2}(1-p)^{(1-a)k}|\mathcal{G}_t|}\leq e^{-e^{-2}(1-p)^{(1-a)k}\left(e^2k^22^{k/2}\right)}\]
\[ = e^{-(1-p)^{(1-a)k}k^22^{k/2}}.\]
By (\ref{abound}), we can bound the exponent by
\[(1-p)^{(1-a)k}k^22^{k/2} \geq (1-p)^{\left(\frac{-\log{\sqrt{2}}}{\log(1-p)}\right)k}k^22^{k/2}\geq k^2.\]
As with $\mathcal{F}_t$, we get that the expected number of sets of cardinality $t$ which are not disjoint from a set of $\mathcal{G}_t$ is less than one. The result follows. 
\end{proof}

\subsection*{Acknowledgements}
The first two authors would like to thank Antonio Gir\~{a}o for many stimulating discussions, one of which lead to the discovery of Example~\ref{Cambridge}.

%%%%%%%%%%%%%%%%%%%%%%%%%%%%%%%%%%%%%%%%%%%%%%%%%%%%%%%
% \bibliographystyle{plain} 
% \bibliography{myBibFile} 
% If you use BibTeX to create a bibliography
% then copy and past the contents of your .bbl file into your .tex file

%\bibliography{firstdraftfull}
 % \bibliographystyle{amsplain}
  
\end{document}